\documentclass[12pt, reqno]{amsart}
\usepackage{amsfonts}
\usepackage{ifthen}
\usepackage{amsthm}
\usepackage{amsmath}
\usepackage{graphicx}
\usepackage{amscd,amssymb,amsthm}
\usepackage[numbers,sort&compress]{natbib}
\newtheorem{theorem}{Theorem}[section]

\newtheorem{corollary}{Corollary}[section]
\newtheorem{remark}{Remark}[section]
\DeclareMathOperator{\RM}{Re}
\begin{document}
\title{On the subclasses  associated with the Bessel-Struve kernel functions}
\author{ Saiful R. Mondal$^\ast$ , Al Dhuain Mohammed} 
\address{ Department of Mathematics,
King Faisal University, Al-Hasa 31982, Saudi Arabia,  \ Email: {\tt smondal@kfu.edu.sa ; albhishi@hotmail.com.com }}
\thanks{$\ast$ Corresponding authors}

\subjclass[2010]{30C45}
\keywords{Bessel-struve operator, Starlike function , convex function ,  Hadamard product .}

\maketitle
\begin{abstract}
The article investigate the necessary and sufficient conditions for the normalized Bessel-struve kernel functions belonging to the classes $\mathcal{T}_\lambda(\alpha)$ , $\mathcal{L}_\lambda(\alpha)$. Some linear operators involving the Bessel-Struve operator are also considered.
\end{abstract}

\section{Introduction}
In this article we will consider the class $\mathcal{H}$ of all analytic functions defined in the unit disk $\mathbb{D}=\{z\in\mathbb{C}:|z|< 1\}$.
Suppose that $\mathcal{A}$ is the subclass of $\mathcal{H}$ consisting of function which are
normalized by $f(0)=0=f'(0)-1$ and also univalent in $\mathbb{U}$. Thus each function $f \in \mathcal{A}$ posses the power series
\begin{align}\label{eqn:analytic-function}
f(z)=z+\sum_{n=2}^\infty a_{n}z^n.
\end{align}
We also consider the  subclass $\mathcal{J}$  of $\mathcal{H}$ consisting of the function have the power series as
\begin{align}\label{eqn:analytic-fun-negative}
f(z)=z-\sum_{n=2}^\infty b_nz^n, \quad b_n >0.
\end{align}
The Hadamard product (or convolution) of two function  $f(z)=z+\sum_{n=2}^\infty a_nz^n \in\mathcal{H}$ and
 $g(z)=z+\sum_{n=2}^\infty b_nz^n\in\mathcal{H}$ is defined  as
 \begin{align*}
(f\ast g)(z)=z+\sum_{n=2}^\infty a_n b_n z^n , z\in\mathbb{D}.
\end{align*}
Dixit and Pal \cite{saad13} introduced the subclass $\mathcal{R}^\tau(A,B)$   consisting  the  function $f\in\mathcal{H}$ which satisfies the inequality
\begin{align*}
\bigg|\frac{f'(z)-1}{(A-B)\tau-B[f'(z)-1]}\bigg|<1,
\end{align*}
for $\tau\in\mathbb{C}\setminus\{0\}$ and $-1\leq B < A \leq 1$. The class $\mathcal{R}^\tau(A,B)$ is the generalization of many well known subclasses, for example, if $\tau=1$ and $A=-B=\alpha \in [0,1)$, then  $\mathcal{R}^\tau(A,B)$ is the subclass of $\mathcal{H}$ which consist the functions satisfying the inequality $\left|\frac{f'(z)-1}{f'(z)+1}\right|< \alpha$  which was studied by  Padmanabhan \cite{saad13}, Caplinger and Causey \cite{saad6} and many others. If the function $f$ of the form $\eqref{eqn:analytic-function}$ belong to the class $\mathcal{R}^\tau (A,B)$, then
\begin{align}\label{eqn:ord-1}
|a_{n}| \leq \frac{(A-B)|\tau|}{n},\quad n\in \mathbb{N}.
\end{align}
The bounds given in $(\ref{eqn:ord-1})$ is sharp.

 In this article we will consider two well-known subclass of $\mathcal{H}$ and denoted as $\mathcal{T}_\lambda ( \alpha)$ and $\mathcal{L}_\lambda ( \alpha)$. The function $f$ in the subclass $\mathcal{T}_\lambda( \alpha)$ satisfy  the analytic criteria
\begin{align*}
\mathcal{T}_\lambda ( \alpha):=\left\{f\in \mathcal{A} :\RM \bigg(\frac{zf'(z)+\lambda z^2f''(z)}{(1-\lambda )f(z)+\lambda zf'(z)}\bigg)> \alpha \right\},
\end{align*}
while  $f \in \mathcal{L}_\lambda ( \alpha)$ have the analytic characterization as
\begin{align*}
\mathcal{L}_\lambda (\alpha):=\left\{f \in \mathcal{A} : \RM \bigg(\frac{\lambda z^3f'''(z)+(1+2\lambda)z^2f''(z) + zf'(z)}{zf'(z)+\lambda z^2f''(z)}\bigg) > \alpha\right\}.
\end{align*}
Here $z \in \mathbb{D}$, $0 \leq \alpha <1$ and $0 \leq \lambda <1$. For more details about this classes and it's generalization see \cite{Altntas, Altntas-2, Altntas-3, HMS} and references therein.  The  sufficient coefficient condition by which a  function $f \in \mathcal{A}$ as defined in \eqref{eqn:analytic-function}
belongs to  the class  $\mathcal{T}_\lambda ( \alpha)$ is
\begin{align}\label{eqn:ord-17}
\sum_{n=2}^{\infty} (n \lambda - \lambda +1)(n- \alpha ) |a_n| \leq 1- \alpha,
\end{align}
and $f$ is in $\mathcal{L}_\lambda ( \alpha)$ is
\begin{align}\label{eqn:ord-18}
\sum_{n=2}^{\infty} n(n\lambda-\lambda+1)(n-\alpha)|a_n|\leq 1-\alpha.
\end{align}
The significance of the class $\mathcal{T}_\lambda(\alpha)$ and  $\mathcal{L}_\lambda(\alpha)$ is that, it will reduce to
some classical subclass of $\mathcal{H}$ for specific choice of $\lambda$. For example, if $\lambda=0$, $\mathcal{T}_\lambda(\alpha)=\mathtt{S}^\ast(\alpha)$ is the class of all starlike functions of order $\alpha$ with respect to the origin
and $\mathcal{L}_\lambda(\alpha)=\mathtt{C}(\alpha)$ is  the class of all starlike functions of order $\alpha$.

Denote $\mathcal{T}^\ast_\lambda(\alpha)=\mathcal{T}_\lambda(\alpha) \cap \mathcal{J}$   and $\mathcal{L}^\ast_\lambda(\alpha)=\mathcal{L}_\lambda(\alpha) \cap \mathcal{J}$. Then \eqref{eqn:ord-17} and \eqref{eqn:ord-18} are respectively
the necessary and sufficient condition for any $f \in \mathcal{J}$ is in $\mathcal{T}^\ast_\lambda(\alpha)$ and $\mathcal{L}^\ast_\lambda(\alpha)$.

Bessel and Struve functions arise in many problems of applied mathematics and mathematical physics. The properties of these functions were studied by many researchers in the past years from many different point of views. Recently the Bessel-Struve kernel and the so-called Bessel-Struve intertwining operator have been the subject of some research from the point of view of the operator theory, see \cite{GS,GS2,KS} and the references therein.
The Bessel-Struve kernel function ${S}_{\nu}$  is defined by the series
$${S}_{\nu}(z)=\sum_{n\geq0}\frac{{\Gamma(\nu+1)}\Gamma{\left(\frac{n+1}{2}\right)}}{\sqrt{\pi} n! \Gamma\left(\frac{n}{2}+\nu+1\right)} z^n,$$
where $\nu>-1.$ This function is  a particular case when $\lambda_1=1$ of the unique solution ${S}_{\nu}(\lambda x)$ of the initial value problem
$$\mathtt{L}_{\nu} u(z)= \lambda_1^2 u(z), \quad u(0)=1, u'(0)=\frac{\lambda_1 \Gamma{(\nu+1)}}{\sqrt{\pi}\Gamma{(\nu+\frac{3}{2}})},$$ where for $\nu>-1/2$ the expression $\mathtt{L}_{\nu}$ stands for the Bessel-Struve operator defined by
$$\mathtt{L}_{\nu}u(z)= \frac{d^2u}{dz^2}(z)+\frac{2\nu+1}{x}\left(\frac{du}{dz}(z)-\frac{du}{dz}(0)\right)$$
with an infinitely differentiable function $u$ on $\mathbb{R}.$

Motivated by the results connecting various subclasses with the Bessel functions \cite{Mondal,saad5}, the Struve functions \cite{Yagmur, Orhan}, the hypergeometric function  and the
Confluent hypergeometric functions \cite{Silverman2, Cho}, in Section \ref{sec2} we obtain several conditions by which the normalized Bessel-Struve kernel functions $zS_\nu$ is the member of $\mathcal{T}_\lambda(\alpha)$ and $\mathcal{L}_\lambda(\alpha)$.  Also determined the condition for which $z(2-S_\nu(z)) \in \mathcal{T}_\lambda^\ast(\alpha)$ and $z(2-S_\nu(z)) \in \mathcal{L}^\ast_\lambda(\alpha)$.

\section{Main Result} \label{sec2}

\subsection{Inclusion properties of the Bessel-Struve functions}
For convenience throughout in the sequel , we use the following notations
\begin{align}\label{eqn:ord-21}
zS_{\nu}( z)= z+\sum_{n=2}^{\infty}c_{n-1}(\nu)z^n,
\end{align}
and
\begin{align}\label{eqn:ord-22}
\Phi(z)=z(2-S_{\nu}( z))= z-\sum_{n=2}^{\infty}c_{n-1}(\nu)z^n,
\end{align}
where $c_n(\nu)=\frac{{\Gamma(\nu+1)}\Gamma{\left(\frac{n+1}{2}\right)}}{\sqrt{\pi} n! \Gamma\left(\frac{n}{2}+\nu+1\right)}$.

Further from ($\ref{eqn:ord-21}$) a calculation yield
\begin{align}\label{eqn:ord-23}
z^2S'_{\nu} ( z)+zS_{\nu}(z)= z +\sum_{n=2}^{\infty}nc_{n-1}(\nu)z^n.
\end{align}
A differentiation of both side of \eqref{eqn:ord-23} two times with respect to $z$,  gives
\begin{align}\label{eqn:ord-24}
 z^3 S''_{\nu}( z)+3z^2S'_{\nu}( z)+zS_{\nu}(z)= z+\sum_{n=2}^{\infty}n^2c_{n-1}(\nu)z^n,
\end{align}
and
\begin{align}\label{eqn:ord-25}
 z^4S'''_{\nu}( z)+6z^3S''_{\nu}( z)+7z^2S'_{\nu}( z)+zS_{\nu}= z +\sum_{n=2}^{\infty}n^3 c_{n-1}(\nu)z^n.
\end{align}
Now we will state and proof our main results.
\begin{theorem}\label{thm:fun-21}
For $\nu> -1/2$ and $0 \leq \alpha, \lambda <1$, let
\begin{align}\label{eqn:ord-26}
\lambda S''_{\nu}(1)+(1-\lambda \alpha)S'_{\nu}(1)+(1- \alpha)S_{\nu}(1)
\leq 2(1- \alpha).
\end{align}
Then the normalized Bessel-Struve function $ zS_{\nu}( z) \in \mathcal{T}_\lambda( \alpha)$.
\end{theorem}
\begin{proof}
Consider  the identity
\begin{align*}
zS_{\nu}( z)= z+\sum_{n=2}^{\infty}c_{n-1}(\nu)z^n.
\end{align*}
Then by virtue of  ($\ref{eqn:ord-17}$) it is enough to show that
$F(n,\lambda, \alpha)\leq1-\alpha$, where
\begin{align*}
F(n,\lambda, \alpha):=\sum_{n=2}^{\infty}(n\lambda-\lambda+1)(n-\alpha)c_{n-1}(\nu)
\end{align*}
The right hand side expression of $F(n,\lambda, \alpha)$ can be rewrite as
\begin{align}\label{eqn-FN}
F(n,\lambda,\alpha)=\lambda\sum_{n=2}^{\infty}n^2 c_{n-1}(\nu)+(1-\lambda(1+ \alpha))
\sum_{n=2}^{\infty}n c_{n-1}(\nu)+ \alpha(\lambda-1)\sum_{n=2}^{\infty}c_{n-1}(\nu).
\end{align}
Now for $|z|=1$, it is evident from \eqref{eqn:ord-21}-- \eqref{eqn:ord-24} that
\begin{align*}
 S_{\nu}(1) &=1+\sum_{n=2}^\infty  z^nc_{n-1}(\nu),\\
S'_{\nu}(1)+S_{\nu}(1)& =1+ \sum_{n=2}^\infty n c_{n-1}(\nu),\\
 S''_{\nu}(1)+ 3S'_{\nu}(1)+S_{\nu}(1)&=1+\sum_{n=2}^\infty  n^2 z^nc_{n-1}(\nu).
\end{align*}
Thus \eqref{eqn-FN} reduce to
\begin{align*}
F(n,\lambda,\alpha)
=\lambda S''_{\nu}(1)+(1-\lambda \alpha+2\lambda)S'_{\nu}(1)+(1- \alpha)(S_{\nu}(1)-1).
\end{align*}
and is bounded above by $1-\alpha$ if \eqref{eqn:ord-26} holds. Thus the proof is completed.
\end{proof}
\begin{corollary}
 The normalized Bessel-Struve kernel  function is starlike  of  order $\alpha \in [0,1)$ with respect to origin if
$S'_{\nu}(1)+(1- \alpha)S_{\nu}(1)
\leq 2(1- \alpha).$
\end{corollary}
\begin{theorem}
 For $\nu> -1/2$ and $0 \leq \alpha, \lambda <1$, suppose that
\begin{align}\label{eqn:ord-27}
\lambda S'''_{\nu}(1)+(5\lambda+1-\lambda \alpha)S''_{\nu}(1)+(4\lambda-2\lambda\alpha-\alpha+3)S'_{\nu}(1)+(1-\alpha)S_{\nu}(1)\leq 2(1- \alpha).
\end{align}
Then the normalized Bessel-Struve function $z S_\nu \in \mathcal{L}_\lambda(\alpha)$.
\end{theorem}
\begin{proof}
By virtue of  ($\ref{eqn:ord-18}$), it
suffices to prove that $G(n,\lambda,\alpha) \leq 1-\alpha,$ where
\begin{align*}
G(n,\lambda,\alpha)&=\sum_{n=2}^\infty n(n\lambda-\lambda+1)(n-\alpha)c_{n-1}(\nu)\\
&=\lambda \sum_{n=2}^\infty n^3 c_{n-1}(\nu)+(1-\lambda(1+\alpha))\sum_{n=2}^\infty n^2 c_{n-1}(\nu)+ \alpha(\lambda-1)\sum_{n=2}^\infty n c_{n-1}(\nu).
\end{align*}
For $|z|=1$  and using ($\ref{eqn:ord-23}$)-($\ref{eqn:ord-25}$), it follows that
\begin{align*}
G(n,\lambda,\alpha)&=\lambda(S'''_{\nu}(1)+7S''_{\nu}(1)+6S'_{\nu}(1)+S_{\nu}(1)-1)\\&\quad \quad +(1-\lambda(1+\alpha))(S''_{\nu}(1)+3S'_{\nu}(1)+S_{\nu}(1)-1)
+\alpha(\lambda-1)(S'_{\nu}(1)+S_{\nu}(1)-1)\\&=\lambda S'''_{\nu}(1)+(5\lambda+1-\lambda\alpha)S''_{\nu}(1)+(4\lambda-2\lambda\alpha-\alpha+3)S'_{\nu}(1)+(1-\alpha)(S_{\gamma,\beta}(1)-1)
\end{align*}
which is bounded above by $1-\alpha$ if ($\ref{eqn:ord-27}$) holds, and hence the conclusion.
\end{proof}
\begin{corollary}
 The normalized Bessel-Struve kernel  function is convex  of of order $\alpha \in [0,1)$ if
\begin{align*}
S''_{\nu}(1)+(3-\alpha)S'_{\nu}(1)+(1-\alpha)S_{\nu}(1)\leq 2(1- \alpha).
\end{align*}
\end{corollary}
\begin{remark}
The condition \eqref{eqn:ord-26} is  necessary and sufficient for
$z(2-S_{\nu}(z))\in\mathcal{T}^*(\alpha),$ while  $z(2-S_{\nu}(z))\in \mathcal{L}^*(\alpha)$ if and only if the condition \eqref{eqn:ord-27} hold.
\end{remark}

\subsection{Operators involving the Bessel-Struve functions and its inclusion properties}

 In this section we considered the linear operator
 $\mathcal{J}_\nu:\mathcal{A} \to \mathcal{A}$
 defined by
 \begin{align*}
 \mathcal{J}_\nu f(z)=zS_{\nu}(z)\ast f(z)=z+\sum_{n=2}^\infty c_{n-1}(\nu)a_nz^n ,
 \end{align*}
In the next result we will study the action of the Bessel-struve operator on the class $\mathcal{R}^\tau(A,B)$.
\begin{theorem}
Suppose that $\nu > -\frac{1}{2}$ and $f\in \mathcal{R}^\tau(A,B)$. For $\alpha, \lambda \in (0,1]$, if
\begin{align}\label{eqn:ord-31}
(A-B)|\tau|\{\lambda S''_{\nu}(1)+(1-\lambda\alpha+2\lambda)S'_{\nu}(1)+(1-\alpha)(S_{\nu}(1)-1)\leq 1-\alpha,
\end{align}
then $\mathcal{J}_{\nu}(f)\in\mathcal{L}_\lambda( \alpha)$.
\end{theorem}
\begin{proof}
Let $f$ be of the form ($\ref{eqn:analytic-function}$) belong to the class $\mathcal{R}^\tau(A,B)$.
By virtue of the inequality $\eqref{eqn:ord-18}$ it suffices to show that
$
G(n,\lambda,\alpha) \leq 1-\alpha.
$

It follows from \eqref{eqn:ord-1} that the coefficient  $a_n$ for each $f \in \mathcal{R}^\tau(A,B)$ satisfy the inequality
$
|a_n|\leq(A-B){|\tau|}/{n},
$
and hence
\begin{align*}
G(n,\lambda,\alpha)
&\leq(A-B)|\tau|\sum_{n=2}^\infty (n\lambda-\lambda+1)(n-\alpha)c_{n-1}(\nu)\\
&= (A-B)|\tau|\left(\lambda\sum_{n=2}^\infty n^2c_{n-1}(\nu)+(1-\lambda(1+\alpha))\sum_{n=2}^\infty n c_{n-1}(\nu)\right.\\
& \left.\quad \quad\quad  + \alpha(\lambda-1)\sum_{n=2}^\infty c_{n-1}(\nu)\right)
\end{align*}
Using \eqref{eqn:ord-23} and \eqref{eqn:ord-24}, the right hand side of the above inequality reduce to
\begin{align*}
G(n,\lambda,\alpha)&=(A-B)|\tau|\{\lambda(S''_{\nu}(1)+3S'_{\nu}(1)+S_{\nu}(1)-1)
    \\ &\quad\quad +(1-\lambda(1+\alpha))(S'_{\nu}(1)+S_{\nu}(1)-1) +\alpha(\lambda-1)(S_{\nu}(1)-1)\}\\
&=(A-B)|\tau|\{\lambda S''_{\nu}(1)+(2\lambda-\lambda\alpha+1)S'_{\nu}(1)+(1-\alpha) (S_{\nu}(1)-1)\}
\end{align*}
   It is evident from the hypothesis $\ref{eqn:ord-31}$ that $G(n,\lambda,\alpha)$  is bounded above by  $1-\alpha$ and hence the conclusion.
 \end{proof}


\begin{theorem}
For $\nu> -1/{2}$, let
$\mathcal{Q}_\nu(z):=\int_{0}^{z}(2-S_{\nu}(t))dt$.
Then for $\alpha, \lambda \in (0,1]$, the operator  $\mathcal{Q}_\nu\in \mathcal{L}^*_\lambda(\alpha)$
if and only if
\begin{align}\label{eqn:ord-33}
\lambda S''_{\nu}(1)+(2\lambda-\lambda\alpha+1)S'_{\nu}(1)+(1-\alpha)S_{\nu}(1)\leq2(1-\alpha).
\end{align}
\end{theorem}

\begin{proof}
Note that
\begin{align*}
\mathcal{Q}_\nu(z)=z-\sum_{n=2}^\infty c_{n-1}(\nu)\frac{z^n}{n}.
\end{align*}
It is enough to show that
\begin{align*}
\sum_{n=2}^\infty n(n\lambda-\lambda+1)(n-\alpha)\left(\frac{c_{n-1}(\nu)}{n}\right)\leq 1-\alpha.
\end{align*}
Since
\begin{align*}
\sum_{n=2}^\infty n(n\lambda-\lambda+1)(n-\alpha)\left(\frac{c_{n-1}(\nu)}{n}\right)=\sum_{n=2}^\infty (n\lambda-\lambda+1)(n-\alpha)c_{n-1}(\nu),
\end{align*}
proceeding as the proof of Theorem $\ref{thm:fun-21}$, we get
\begin{align*}
\sum_{n=2}^\infty n(n\lambda-\lambda+1)(n-\alpha)\left(\frac{c_{n-1}(\nu)}{n}\right)=\lambda S''_{\nu}(1)+(2\lambda-\lambda\alpha+1)S'_{\nu}(1)+(1-\alpha)S_{\nu}(1),
\end{align*}
which is bounded above by $1-\alpha$ if and only if ($\ref{eqn:ord-33}$)
\end{proof}

\end{document}